\documentclass[11pt]{article} 

\usepackage[utf8]{inputenc} 
\usepackage{epstopdf}

\usepackage{geometry} 
\usepackage{amsmath}
\allowdisplaybreaks
\usepackage{graphicx} 
\usepackage{mathtools}

\usepackage{float}
\usepackage{setspace}
\usepackage{booktabs} 
\usepackage{array} 
\usepackage{paralist} 
\usepackage{amsfonts}
\usepackage{verbatim} 
\usepackage{subfig} 
\usepackage{amsthm}
\usepackage{color}
\usepackage{graphicx}
\usepackage{fancyhdr} 
\usepackage{sectsty}
\usepackage{indentfirst}
\allsectionsfont{\sffamily\mdseries\upshape} 
\usepackage[mathscr]{euscript}
\usepackage[nottoc,notlof,notlot]{tocbibind} 
\usepackage[titles,subfigure]{tocloft} 

    \newtheorem{theorem}{Theorem}
    
    \newtheorem{proposition}[theorem]{Proposition}

    \newenvironment{definition}[1][Definition]{\begin{trivlist}
    \item[\hskip \labelsep {\bfseries #1}]}{\end{trivlist}}

\title{Partially Observed, Multi-objective Markov Games}
\author{Yanling Chang, Alan L. Erera, and Chelsea C. White III\\\\
School of Industrial and Systems Engineering\\
Georgia Institute of Technology}
\date{}
\begin{document}
\maketitle
\doublespacing
\begin{abstract}

The intent of this research is to generate a set of non-dominated policies from which one of two agents (the leader) can select a most preferred policy to control a dynamic system that is also affected by the control decisions of the other agent (the follower). The problem is described by an infinite horizon, partially observed Markov game (POMG). The actions of the agents are selected simultaneously at each decision epoch. At each decision epoch, each agent knows: its past and present states, its past actions, and noise corrupted observations of the other agent's past and present states. The actions of each agent are determined by the agent's policy, which selects actions at each decision epoch based on these data.

The leader considers multiple objectives in selecting its policy. The follower considers a single objective in selecting its policy with complete knowledge of and in response to the policy selected by the leader. This leader-follower assumption allows the POMG to be transformed into a specially structured, partially observed Markov decision process (POMDP). This POMDP is used to determine the follower's best response policy. A multi-objective genetic algorithm (MOGA) is used to create the next generation of leader policies based on the fitness measures of each leader policy in the current generation. Computing a fitness measure for a leader policy requires a value determination calculation, given the leader policy and the follower's best response policy. The policies from which the leader can select a most preferred policy are the non-dominated policies of the final generation of leader policies created by the MOGA. An example is presented that illustrates how these results can be used to support a manager of a liquid egg production process (the leader) in selecting a sequence of actions to best control this process over time, given that there is an attacker (the follower) who seeks to contaminate the liquid egg production process with a chemical or biological toxin.
\end{abstract}
\newpage
\section{Introduction}
\renewcommand{\labelitemi}{$\bullet$}
The intent of this research is to provide decision support to a \textit{leader}, an agent that wishes to determine a policy that will select actions to best control a sequential stochastic system over an infinite planning horizon, given that there is a \textit{follower}, an agent that also would like to exert control over the system for its own purposes. We assume:
\begin{itemize}
\item At each decision epoch, each agent knows: its past and present states, its past actions, and  noise corrupted observations of the other agent's past and present states.
\item The leader's and follower's actions are selected simultaneously at each decision epoch.  \item Each agent's policy selects actions based on data currently available to the agent.
\item The follower knows the leader's policy and determines a response policy that is optimal with respect to the follower's objective.
\item The leader considers multiple objectives in selecting its policy.  
\end{itemize}
Our objective is to find a set of non-dominated policies from which the leader can select its most preferred policy. Such information can serve as input to a decision support system that, for example, is based on a deterministic version of multi-attribute utility theory (Keeney \& Raiffa, 1993; Holloway \& White, 2008). In this context, the results presented in this paper generate options (i.e., policies) for consideration by (1) creating multiple generations of policies and eliminating all but the non-dominated set of policies from the last generation and (2) determining value scores for each of the policies in this non-dominated set.

We remark that the assumption that the follower knows the policy of the leader is a conservative assumption from the perspective of the leader and could unrealistically bias the game to the advantage of the follower.  However, this bias is mollified by the fact that the leader and follower do not share the same data at each decision epoch, and hence the follower can only infer what action the leader will actually take.  This assumption is also reasonable for many applications.  For example, if the leader is a large governmental agency or corporation and the follower is an individual or group intent on attacking the leader, then it may be reasonable to assume that the follower will know more about the leader than the leader will know about the follower.  Further, the assumption that the follower knows the policy of the leader allows us to transform the Markov game that we use to model leader-follower interaction into a model of sequential decision making under uncertainty and hence to take advantage of this computationally useful transformation.

The motivating application of this research is the operation of a liquid egg production facility in order to maximize the supply chain's productivity while minimizing its vulnerability to the intentional insertion of a biological or chemical toxin into the food production and distribution system (see Manning, 2005; O'Ryan, 1996; Sobel, 2002).  For this application, we assume the leader manages the production facility and is trying to balance two objectives: (1) maximize productivity and (2) minimize vulnerability. See Mohtadi \& Murshid (2009) for background information about this application area. Although initially developed to model this application, we remark that the decision support process to be presented can model a particularly broad class of sequential game applications, when all agents are intelligent and adaptive.

Many of the methodological characteristics of the decision support model presented in this paper have been considered elsewhere in the decision, risk, and reliability analysis literatures. Models of intelligent agents or adversaries are examined by Cardoso \& Diniz (2009). The single-period leader-follower game has been widely used to analyze the strategic interactions between two intelligent and adaptive agents. For example, Cavusoglu et al. (2013) have studied the impacts of passenger profiling on airport security operations. Bakir (2011) analyzed resource allocation for cargo container transportation security. Other applications of the single-period leader-follower game are presented in Bier (2007, 2008) and Zhuang \& Bier (2007). Multi-period games have been considered by Wang \& Bier(2011), who examined a two-period leader-follower repeated game, and by Hausken \& Zhuang (2011), who studied a multi-period game with myopic agents. Application of completely observable stochastic game to overseas cargo container security can be found in Bakir \& Kardes (2009) to capture the state dynamics over time.  Models that consider incomplete or uncertain information are presented and analyzed by McLay (2012), Rothschild (2012) and Wang \& Bier (2011). 

Each of above methodological characteristics is intended to enhance the realism of the respective model.  Our model extends the existing literature on sequential games by explicitly considering the multi-period interaction of two non-myopic agents, each of whom adjust its decisions according to the other agent's decisions over an infinite planning horizon. Furthermore, our model considers the case where neither agent has complete information about the other agent. By combining these characteristics into a single model, as has been done in this paper, we believe that the modeling realism of the resulting model has been further enhanced. However, and not surprisingly, additional model realism has resulted in increased computational challenges.  Dealing with these challenges is the focus of much of this paper.

Our approach to decision support is described as follows.  We begin with an initial (i.e., first generation) set of possible leader policies.  We then use a multi-objective genetic algorithm (MOGA) to create successive generations of leader policies.  Presumably, the next generation of leader policies contains, in some sense, higher quality policies than the current generation.  We then determine the non-dominated set of the last generation of leader policies determined and present this set to the leader.  The leader can then select the most preferred policy from this set for implementation.  

Mimicking the process of natural evolution, the MOGA creates the next generation of policies from the current set based on the fitness measures of each of the policies in the current set.  Each fitness measure is related to an objective of the leader.  We model the interaction between the leader and follower as a partially observed Markov game (POMG). Our POMG is a version of the partially observed stochastic game (POSG) where the state dynamics possess the Markovian property. We assume that the follower is aware of the policy that the leader has selected and makes use of this fact in constructing the follower's policy.  Thus at the policy level, the game is a leader-follower (Stackelberg) game.  This assumption allows the POMG to be converted into a partially observed Markov decision process (POMDP).  The fitness measures for each leader policy in the current generation are needed by the MOGA to create the next generation of leader policies from the current generation. These fitness measure are computed by a value determination procedure, given the leader policy and the follower response policy.

The paper is organized as follows.  We review the pertinent literature associated with the MOGA, POSG, and POMDP in Section 2.   In Section 3, we describe the MOGA in more detail, show how the fitness measures are computed using the POMG and the POMDP, present equilibrium conditions, and discuss the value of information. The numerical evaluation in Section 4 applies this decision support procedure to a simplified liquid eggs supply chain security problem and analyses the value of information for the agents. Section 5 summarizes research results and discusses future research directions.

\section{Literature Review} 
The research presented in this paper combines and extends results associated with POMDP, POSG, and MOGA.  We now review the pertinent literature in these three areas of research.

\subsection{The partially observed Markov decision process}
The POMDP is a model of sequential decision making under uncertainty that takes into consideration noise corrupted and/or costly observations of the state of the system under control.  Relative to the completely observed Markov decision process (i.e., the MDP; see Puterman, 1994), the POMDP represents a more general but significantly more computationally challenging model.  In seminal research, Smallwood and Sondik (1973) and Sondik (1978) showed that under robust conditions the optimal cost function for both the finite horizon and infinite horizon expected total discounted cost criterion POMDPs is piecewise linear and convex, and presented successive approximations approaches for solving these POMDPs that exploited this structure. Zhang (2010) revisited these structural results and convergence properties by exploiting the dual relationship between hyperplanes and points in the POMDP and related the solution of the POMDP to the Minkowski sum problem in computational geometry. Monahan (1982), Eagle (1984), and White and Scherer (1989) presented improved algorithms based on the structural results.  Detailed descriptions of other exact algorithms can be found in Cheng (1988), Littman (1994b), Cassandra (1994a), Cassandra, Littman and Zhang (1997),  Feng and Zilberstein (2004), Lin and White (1998, 2004) and Naser-Moghadasi (2012). Surveys of related solution techniques and complexity analyses for the POMDP can be found in Monahan (1982), Lovejoy (1991), White (1991), Cassandra (1994b) and Poupart (2005). 

In the development of approximate solution techniques for POMDPs, point-based value iteration (PBVI) is presented and analysed in Pineau (2003) and Shani (2012).  Platzman (1977, 1980), White and Scherer (1994), Littman (1994a),  Hauskrecht (1997), Hansen (1998a, 1998b),  Poupart (2005) examined finite memory policy and finite-state controllers. Varakantham (2007) and Poupart (2011) focused on calculating bounds on optimal POMDP solutions in order to evaluate the quality of approximate solutions.  Surveys of approximation methods for POMDPs can be found in Hauskrecht (2000), Aberdeen (2003), and Yu (2007). Approximate algorithms have proved useful for large-scale problems (Hoey, 2010; Thomson and Young, 2010).
\subsection{The partially observable stochastic game}
The stochastic game introduced by Shapley (1953) represents a multi-agent planning problem in a stochastic environment. In this setting, each player considers the consequences of its own action and the actions that its opponents or teammates may take. See Raghavan and Filar (1991), Filar (1997) and Ummels (2010) for details. 

The POSG is a new, relatively unexamined generalization of the stochastic game, where the states of the game are not precisely observed by the players and all players make decisions based on these partial observations.  Although POSGs provide a robust framework for multi-agent planning, Bernstein (2002) showed that POSGs are computationally intractable when problem size grows. Rabinovich (2003) has shown that even epsilon-optimal approximations are NP-hard.  As a result, POSGs with special structures that enhance computational tractability are of considerable interest.  Koller (1994) provided an efficient algorithm for a two-player POSG with tree-like structure.  McEneaney (2004) focused on a game where only one player has imperfect information.  Ghosh (2004), Oliehoek (2005), and Bopardikar (2011) studied a zero-sum version of the POSG. Emery-Montemerlo (2004) approximated POSGs with common payoffs by a series of Bayesian games.  A cooperative version of the POSG, called a decentralized partially observable Markov decision process (DEC-POMDP), has been studied by Becker (2004), Bernstein (2005), Seuken (2007), and Oliehoek (2008). A survey of the DEC-POMDP can be found in Oliehoek (2012).

It is often the case for real-world planning problems that the players' payoffs are neither completely aligned with others nor directly opposed.  Hespanha and Prandini (2001) proved the existence of Nash equilibrium in a two-player finite-horizon POSG. Hansen (2004) developed a dynamic program for general POSGs by pruning very weakly dominated strategies and then showed that this dynamic programming approach can achieve optimality for cooperative settings.  However, this approach is computationally infeasible for all but the smallest problems. Kumar and Zilberstein (2009) developed an approximate solution procedure for the POSG based on Hansen's work.  Interactive POMDPS addressed in Gmytrasiewicz and Doshi (2005) demonstrated another framework for multi-agent planning.

\subsection{The multi-objective genetic algorithm}
Genetic algorithms, introduced by Holland (1975), are adaptive heuristic search techniques that mimic the process of natural evolution.  A genetic algorithm represents each feasible problem solution in a population of solutions as a genome or chromosome, and begins with an initial population of feasible solutions.  Solutions having high measures of fitness are preferably selected during each generation to produce the next generation of solutions having improved fitness measures by applying genetic (e.g., mutation and crossover) operators. After a number of generations, the population presumably evolves to optimal or near-optimal solutions. Goldberg (1989), Forrest (1993) and Srinivas(1994) present surveys of genetic algorithms and the theories.

Multi-objective genetic algorithms (MOGA) are designed for the simultaneous optimization of multiple, often competing objectives. Usually the optimal solutions are a set of points, called the Pareto-optimal set, in the sense that no improvement can be made in any objective without sacrificing the other objectives. MOGA pushes the Pareto frontier towards the ideal optimal set of solutions as the algorithm proceeds. MOGA algorithms include: the vector evaluated GA (VEGA) (Schaffer, 1985), the Niched Pareto GA (NPGA) (Horn, 1994), the Pareto Envelope-based Selection Algorithms (PESA) (Corne, 2000) and the Fast Non-dominated sorting GA (NSGA-II) (Deb, 2002). Surveys are presented in Coello (2000) and Konak (2006).  MOGAs have been widely applied in optimization and decision making problems (see Ponnambalam, 2001; Deb, 2001; Ombuki, 2006; Lin, 2008; Bowman, 2010; Yildirim, 2012).
 
\section{Model and Analysis}
We first present the POMG model in Section 3.1. In order to determine an optimal response policy for the follower, the POMDP is constructed by combining the POMG model with any leader's policy. The resulting POMDP is presented and examined in Section 3.2. For computational reasons, we require that the leader and follower policies be finite memory policies. However, the POMDP constructs a perfect memory follower policy.  In Section 3.3, we present an approach for determining a finite-memory approximation of a perfect memory policy.  In order for the MOGA to determine the next generation of leader policies, fitness measures must be calculated for each policy in the current generation of leader policies.  Each fitness measure is associated with an objective of the leader.  In Section 3.4 we present an approach for determining the fitness measures, for any given leader policy and follower policy. Section 3.5 shows how to use MOGA to generate a non-dominate set of leader policies. Section 3.6 and 3.7 address equilibria and the value of information, respectively. 

\subsection{Partially Observed Markov Game}
The partially observed Markov game (POMG) serves as the modeling basis of our decision support system design.  The POMG is comprised of:

\textit{Decision epochs:} Let $\{0, 1, …\}$ be the set of all decision epochs when both agents select actions simultaneously.  Thus, the problem horizon is countable and infinite.

\textit{State spaces:} Let $S^L$ and $S^F$ be the state spaces of the leader and the follower, respectively.  Both spaces are epoch-invariant and finite.  At decision epoch t, let $s^L(t)$ be the leader's state, $s^F(t)$ be the follower's state, and denote $s(t) = \{s^L(t), s^F(t)\}$.

\textit{Action spaces:} Let $A^L$ and $A^F$ be the epoch-invariant action spaces of the leader and the follower, both of which are finite.  At decision epoch $t$, let $a^L(t)$ be the leader’s action, $a^F(t)$ be the follower's action, and denote $a(t) = \{a^L(t), a^F(t)\}$.

\textit{Observation spaces:}  Let $Z^L$ and $Z^F$ be the observation spaces of the leader and the follower, both of which are epoch-invariant and finite.  At decision epoch $t$, let $z^F(t)$ be the follower's observation of the leader's state, $z^L(t)$ the leader's observation of the follower's state, and denote $z(t) = \{z^L(t), z^F(t)\}$.

\textit{Systems dynamics:}  We assume the epoch-invariant probability $P(z(t+1), s(t+1) | s(t), a(t))$ is given.  Note that $$P(z(t+1), s(t+1) | s(t), a(t)) = P(z(t+1) | s(t+1), s(t), a(t)) P(s(t+1) | s(t), a(t)),$$ where $P(z(t+1) | s(t+1), s(t), a(t))$ is referred to as the state observation probability and $P(s(t+1) | s(t), a(t))$ is referred to as the state transition probability. 

\textit{Information patterns:}  The information pattern for agent $k $ describes what agent $k$ knows and when agent $k$ knows it.  Let
\begin{enumerate}[$\bullet$]
\item $Z^k(t) = \{z^k(t), z^k(t – 1), …\}$ and $Z^k(t, \tau) = \{z^k(t), z^k(t – 1), …, z^k(t – \tau + 1)\}$
\item $S^k(t) = \{s^k(t), s^k(t – 1), …\}$ and $S^k(t, \tau) = \{s^k(t), s^k(t – 1), …, s^k(t – \tau + 1)\}$
\item $\mathscr{A}^k(t) = \{a^k(t - 1), a^k(t – 2), …\}$ and $\mathscr{A}^k(t, \tau) = \{a^k(t - 1), …, a^k(t – \tau)\}$
\item $\mathscr{I}^k(t) = \{ \mathscr{Z}^k(t),\mathscr{S}^k(t),  \mathscr{A}^k(t)\}$
\item $\mathscr{I}^k(t, \tau) = \{  \mathscr{Z}^k(t, \tau),\mathscr{S}^k(t, \tau), \mathscr{A}^k(t, \tau)\}$ 
\end{enumerate}
We assume that agent $k$ chooses $a^k(t)$ on the basis of $\mathscr{I}^k(t)$, if agent $k$ has perfect memory, or on the basis of $\mathscr{I}^k(t, \tau)$, if agent $k$ has finite memory. Note that $\mathscr{I}^k(t)=\{\mathscr{I}^k(t, \tau), \mathscr{I}^k(t-\tau)\}.$

\textit{Single Period Cost and Criteria:}  Let $c^F(s(t), a(t))$ be the decision epoch invariant single period cost accrued by the follower at epoch $t$, given $s(t)$ and $a(t)$, and let $c^L_i(s(t), a(t))$ be the decision epoch invariant single period cost accrued by the leader with respect to criterion $i$ at epoch $t$, given $s(t)$ and $a(t)$.  The criteria under consideration are the concomitant expected total discounted costs over the infinite horizon.

\textit{Policies:}  A policy $\pi^k$ for agent $k$ is a mapping from what agent $k$ knows at epoch $t$, either $\mathscr{I}^k(t)$ or $\mathscr{I}^k(t, \tau)$, into its set of available actions, $A^k$.  Policies can be random and hence described by conditional probabilities.  We restrict our interest to stationary policies. Stationary policies tend to be easy to implement and in many cases, e.g., the determination of optimal follower response policies for a broad class of scalar criteria, sufficiently rich to contain an optimal policy.

\textit{Objectives:}  The follower's objective is to select a stationary policy that minimizes its criterion. The leader's objective is to optimize all criteria under consideration in some balanced manner, with this balance being determined by the leader.  Our objective is to provide the leader with a non-dominated set of policies from which to choose a single policy for implementation.

\subsection{Determination of a Best Response Policy $\bar{\pi}^F$, Given a leader policy $\pi^L$}
Let $\bar{\pi}^F$ be the \textit{perfect} memory best response policy to the leader policy $\pi^L$. We assume that the leader policy $\pi^L = \{P(a^L(t) | \mathscr{I}^L(t, \tau))\}$ and that the follower knows $\pi^L$.  We also assume that the follower knows $\mathscr{I}^F(t)$.  (We will restrict the follower's information pattern to $\mathscr{I}^F(t, \tau)$ below; however, for the moment it will be convenient to assume that the follower has perfect memory.)  We remark that although the follower knows the leader's policy, because the information patterns of the agents are in general different, the follower can only infer what action the leader will actually take.

Let $v^F(\pi^L; \mathscr{I}^F(t))$ be the follower's optimal criterion value, given $\mathscr{I}^F(t)$ and $\pi^L$.  For notational simplicity, assume the dependence of $\pi^L$ is implicit; hence, $v^F(\mathscr{I}^F(t)) = v^F(\pi^L; \mathscr{I}^F(t))$.  Then, according to results in (Puterman, 1994; Chapter 6), $v^F$ uniquely satisfies 
\begin{equation}
v^F = H^Fv^F
\end{equation}
where for any $v$,
$$[H^Fv](\mathscr{I}^F(t)) = \min_{a^F(t)} h^F(\mathscr{I}^F(t), a^F(t), v),$$ 
$$h^F(\mathscr{I}^F(t), a^F(t), v) = E\{c^F(s(t), a(t)) + \beta v(\mathscr{I}^F(t+1)) | \mathscr{I}^F(t), a^F(t)\},$$
and where $E$ is the expectation operator and the minimum is over all $a^F(t)$.  We now state our first result.

\begin{proposition}
\label{prop1}
Assume $\pi^L$ is given. Then for each $s^F(t)$ there is an at most countable set of arrays $\Gamma^*(s^F(t))$ such that:

$$v^F(\mathscr{I}^F(t)) = \min\{\sum \gamma(\mathscr{I}^L(t, \tau)) P(\mathscr{I}^L(t, \tau)| \mathscr{I}^F(t)): \gamma \in \Gamma^*(s^F(t))\},$$

where the sum is over all $\mathscr{I}^L(t, \tau)$.

\end{proposition}

\begin{proof}
Assume $v$ and $\Gamma$ are such that
$$v(\mathscr{I}^F(t)) = \min\{\sum \gamma(\mathscr{I}^L(t, \tau)) P(\mathscr{I}^L(t, \tau)| \mathscr{I}^F(t)): \gamma \in  \Gamma(s^F(t))\},$$
where the sum is over all $I^L(t, \tau)$.  Then straightforward analysis, following arguments in (Smallwood and Sondik, 1973), shows that
$$h^F(\mathscr{I}^F(t), a^F(t), v) = \min\{\sum \gamma'(\mathscr{I}^L(t, \tau)) P(\mathscr{I}^L(t, \tau)| \mathscr{I}^F(t)): \gamma' \in \Gamma’(s^F(t), a^F(t))\},$$

where the sum is over all $\mathscr{I}^L(t, \tau)$, where if $\gamma ' \in \Gamma'(s^F(t), a^F(t))$ then $\gamma'$ is of the form
\begin{align*}
&\gamma’(\mathscr{I}^L(t, \tau)) = \sum_{a^L(t)}P(a^L(t)| \mathscr{I}^L(t, \tau))[c^F(s(t), a(t)) \\
&+ \beta \sum_{s(t+1)} \sum_{z(t+1)} \gamma^{i,j}(z^L(t+1), s^L(t+1), a^L(t), \mathscr{I}^L(t, \tau - 1)) P(z(t+1), s(t+1)| s(t), a(t))],
\end{align*}
where $\gamma^{i,j}$ can be any element in $\Gamma(s^F(t+1))$ for each $s^F(t+1)=i$ and $z^F(t+1)=j$. Then, $$[H^Fv](\mathscr{I}^F(t))=\min \{\sum_{\mathscr{I}^L(t,\tau)} \gamma''(\mathscr{I}^L(t,\tau))P(\mathscr{I}^L(t,\tau)|I^F(t)): \gamma'' \in \Gamma''(s^F(t))\},$$
where $\Gamma''(s^F(t))=\cup_{a^F(t)}\Gamma'(s^F(t),a^F(t)).$

The operator $H^F$ is a contraction operator on the Banach space comprised of all functions mapping $\mathscr{I}^F(t)$ into the real line, having as its norm the supremum norm, and as a result, the sequence $\{v^{n}\}$, where $v^{n+1} = H^Fv^{n}$, converges to $v^F$ for any given $v^{0}$.  The above result indicates that $H^F$ preserves piecewise linearity and concavity and in the limit preserves concavity. 
\end{proof}

We remark that $\Gamma''(s^F(t))$ usually contains many redundant vectors, where $\gamma$ is redundant if $[H^Fv](\mathscr{I}^F(t))$ is strickly less than $\sum_{\mathscr{I}^L(t,\tau)}\gamma(\mathscr{I}^L(t,\tau))P(\mathscr{I}^L(t,\tau)|\mathscr{I}^F(t))$ for all $\{P(\mathscr{I}^L(t,\tau)|\mathscr{I}^F(t))\}$. From both storage and computational perspectives, there is a value to keep the cardinality of $\Gamma''(s^F(t))$ as small as possible. Let the operator \textbf{PURGE} be such that $\Gamma'''(s^F(t))=\textbf{PURGE}(\Gamma''(s^F(t)))$ is the subset of $\Gamma''(s^F(t))$ having the smallest cardinality that satisfies \begin{align*} &\min\{\sum_{I^L(t,\tau)} \gamma''(\mathscr{I}^L(t,\tau)P(\mathscr{I}^L(t,\tau)|\mathscr{I}^F(t))): \gamma '' \in \Gamma''(s^F(t+1))\} \\&= \min\{\sum_{I^L(t,\tau)} \gamma'''(\mathscr{I}^L(t,\tau)P(\mathscr{I}^L(t,\tau)|\mathscr{I}^F(t))): \gamma '' \in \Gamma'''(s^F(t+1))\}\end{align*}
for all $\{P(\mathscr{I}^L(t,\tau)|\mathscr{I}^F(t))\}$. Related discussion about the necessity and the existence of the \textbf{PURGE} operator can be found in Lin and White (1998). 

With regard to the implications of Proposition \ref{prop1} and results in (Puterman, 1994; Chapter 6), $v^F$ and hence an optimal policy can depend on $\mathscr{I}^F(t)$ only through $(s^F(t), y^F(t))$, where $y^F(t) = \{P(\mathscr{I}^L(t, \tau)| \mathscr{I}^F(t)): \mbox{ for all } \mathscr{I}^L(t, \tau)\}$.  Hence, $(s^F(t), y^F(t))$ is a sufficient statistic.  Further, $v^F$ is concave in $y^F(t)$.  Additionally, if $\gamma^*(s^F(t))$ is a finite set of arrays for all $s^F(t)$, then $v^F$ is piecewise linear.  Note that the dimension of $(s^F(t), y^F(t))$ is finite and $t$-invariant.  Note further that the finite dimensionality of $y^F(t)$ follows directly from the finite-memory assumption imposed on $\pi^L$.  Thus, assuming $\gamma^*(s^F(t))$ is a finite set of arrays and $v^F$ is described in terms of $(s^F(t), y^F(t))$, $v^F$ has a finite representation.  We remark that the cardinality of $\Gamma'(s^F(t))$ can be substantially larger than the cardinality of $\Gamma(s^F(t))$, where both $\Gamma(s^F(t))$ and $\Gamma'(s^F(t))$ are defined in the proof of Proposition \ref{prop1}.  Techniques for reducing the cardinality of $\Gamma'(s^F(t))$ can be found in White (1991).  

\subsection{A Finite-Memory Approximation to $\bar{\pi}^F$}
As noted above, an optimal policy $\bar{\pi}^F$ for the follower that achieves the minimum in Equation 1 depends on $\mathscr{I}^F(t)$ and hence $\bar{\pi}^F$ is a perfect-memory policy. In order to insure that the leader criteria have finite representation, the follower policy must also be a finite-memory policy.  We determine a finite-memory (approximate) policy from a given perfect memory policy as follows.  We note that $\{(\mathscr{I}^F(t, \tau) , y^F(t - \tau)), t = 1, 2, …\}$ is also a sufficient statistic for this problem, with $y^F(t - \tau)$ representing the influence of data determined up through epoch $t – \tau$.  By (arbitrarily) assuming a uniform distribution over $y^F(t - \tau)$, we determine probabilities of the form $P(a^F(t) | \mathscr{I}^F(t, \tau))$.  Our numerical analyses indicate that these finite memory approximations of optimal perfect memory policies can be remarkably accurate, even for small $\tau$.  Identifying and analyzing other approaches for determining a finite-memory policy from a perfect memory policy is a topic for future consideration. In the following context, we use $\pi^F$ to denote a \textit{finite} memory best response policy and $\bar{\pi}^F$ to represent a \textit{perfect} memory best response policy of the follower.  

We remark that an alternative approach for directly determining a finite-memory follower policy in response to a given leader policy involves determining $v^F$ as a function of $\mathscr{I}^F(t, \tau)$, rather than as a function of $\mathscr{I}^F(t)$.  Whether or not such an approach could be useful is a topic of future research; see Platzman (1977) and White (1994) for related discussion.  

\subsection{Fitness Measure Determination}
Let $v^L_i(\pi^L, \pi^F; \mathscr{I}^L(t))$ be the criterion value for the leader's $i^{th}$ criterion, given $\mathscr{I}^L(t)$, a leader policy $\pi^L = \{P(a^L(t) | \mathscr{I}^L(t, \tau))\}$, and a follower policy $ \pi^F= \{P(a^F(t) | \mathscr{I}^F(t, \tau))\}$.  For notational simplicity, we assume that the dependence of $v^L_i(\pi^L, \pi^F; \mathscr{I}^L(t))$ on $(\pi^L, \pi^F)$ is implicit; hence, $v^L_i(\mathscr{I}^L(t)) = v^L_i(\pi^L, \pi^F; \mathscr{I}^L(t))$.  Then according to results in (Puterman, 1994; Chapter 6), $v^L_i$ uniquely satisfies
$$v^L_i(\mathscr{I}^L(t)) = h^L_i(\mathscr{I}^L(t), v^L_i)$$
for all $\mathscr{I}^L(t)$, where
$$h^L_i(\mathscr{I}^L(t),v)=E\{c^L_i(s(t),a(t))+\beta v(\mathscr{I}^L(t+1))|\mathscr{I}^L(t) \}.$$

We now show that $v^L_i(\mathscr{I}^L(t))$ is dependent on $\mathscr{I}^L(t)$ only though $(\mathscr{I}^L(t, \tau), y^L(t))$, where the array $y^L(t) = \{P(I^F(t, \tau)| \mathscr{I}^L(t)): \mbox{ for all } \mathscr{I}^F(t, \tau)\}$.  Thus, $(\mathscr{I}^L(t, \tau), y^L(t))$ is a sufficient statistic.  

\begin{proposition}
\label{prop2}
Assume $(\pi^L, \pi^F)$ are given as two finite-memory policies. Then, there is a function $g^*_i$ such that $$v^L_i(\mathscr{I}^L(t))=\sum g_i^*(\mathscr{I}^L(t, \tau), \mathscr{I}^F(t, \tau))P(\mathscr{I}^F(t, \tau)|\mathscr{I}^L(t)),$$
where the sum is over all $\mathscr{I}^F(t,\tau)$. Further, $g^*_i$ is the unique solution of the equation
\begin{align*}
&g^*_i(\mathscr{I}^L(t, \tau), \mathscr{I}^F(t, \tau)) =  \sum\nolimits^1 P(s(t),a(t)|\mathscr{I}^F(t, \tau),\mathscr{I}^L(t, \tau))\{c^L_i(s(t),a(t))\\
&+\beta \sum\nolimits^2 g^*_i[(\mathfrak{z}^L(t+1), \mathscr{I}^L(t, \tau-1)),(\mathfrak{z}^F(t+1), \mathscr{I}^F(t, \tau-1))]P(z(t+1),s(t+1)|s(t),a(t)) \},
\end{align*}
where $\mathfrak{z}^k(t)=\{z^k(t),s^k(t),a^k(t-1) \}$, $\sum^1$ is over all $s(t)$ and $a(t)$, and $\sum^2$ is over all $z(t+1)$ and $s(t+1)$.
\end{proposition}
\begin{proof}
We remark that since $(\pi^L, \pi^F)$ is assumed given, $P(s(t),a(t)|\mathscr{I}^F(t,\tau), \mathscr{I}^L(t,\tau))$ is well defined. Assume there is a function $g$ such that $$v(\mathscr{I}^L(t)) = \sum g(\mathscr{I}^L(t,\tau),\mathscr{I}^F(t,\tau))P(\mathscr{I}^F(t,\tau)|\mathscr{I}^L(t)),$$
where the sum is over all $\mathscr{I}^F(t,\tau)$. Then it is straightforward to show that there is a function $g'$ such that
$$h^L_i(\mathscr{I}^L(t),v) = \sum g'(\mathscr{I}^L(t,\tau), \mathscr{I}^F(t,\tau))P(\mathscr{I}^F(t,\tau)|\mathscr{I}^L(t)),$$
where the sum is over all $\mathscr{I}^F(t,\tau)$, and  
\begin{align*}
&g'(\mathscr{I}^L(t, \tau), \mathscr{I}^F(t, \tau)) =  \sum\nolimits^1 P(s(t),a(t)|\mathscr{I}^F(t, \tau),\mathscr{I}^L(t, \tau))\{c^L_i(s(t),a(t))\\
&+\beta \sum\nolimits^2 g[(\mathfrak{z}^L(t+1), \mathscr{I}^L(t, \tau-1)),(\mathfrak{z}^F(t+1), \mathscr{I}^F(t, \tau-1))]P(z(t+1),s(t+1)|s(t),a(t)) \},
\end{align*}
and where $\mathfrak{z}^k(t)=\{z^k(t),s^k(t),a^k(t-1) \}$, $\sum^1$ is over all $s(t)$ and $a(t)$, and $\sum^2$ is over all $z(t+1)$ and $s(t+1)$.The result follows directly from the following facts:
\begin{enumerate}[$\bullet$]
\item The operator $H^L$, where $[H^Lv](\mathscr{I}^L(t)) = h^L_i(\mathscr{I}^L(t), v)$, is a contraction operator on the Banach space comprised of all functions mapping $\mathscr{I}^L(t)$ into the real line, having as its norm the supremum norm.
\item As a result, the sequence $\{v^{n}\}$, where $v^{n+1} = H^Lv^{n}$, converges to $v^L$ for any given $v^{0}$.
\end{enumerate}
\end{proof}
Since both $\pi^L$ and $\pi^F$ are finite-memory policies, then both $\mathscr{I}^L(t, \tau)$ and $y^L(t)$ are $t$-invariant arrays of finite dimension, which enhances the potential computability of $v^L$.We remark that Proposition \ref{prop2} holds for any given finite memory leader policy $\rho^L$ and follower policy $\rho^F$, and hence $\rho^F$ is not necessarily a response policy to $\rho^L$.

We now summarize how fitness measures are determined for a given finite-memory leader policy:
\begin{enumerate}[$\bullet$]
\item	Step 1: Determine a \textit{perfect} memory follower response policy that achieves the minimum expected cost for the follower, using Proposition \ref{prop1}.
\item	Step 2: Approximate the resulting perfect-memory follower response policy by a \textit{finite}-memory policy.
\item	Step 3: Given the leader policy and the follower's approximate response policy, determine the concomitant fitness measures using Proposition \ref{prop2}.
\end{enumerate}

\subsection{Multi-Objective Genetic Algorithm}
We now describe how we use a multi-objective genetic algorithm (MOGA), NSGA-II (Deb, 2002), to generate policies from which the leader will choose a most preferred policy.

Let $\{\pi^L(m), m = 1, …, M\}$ be the current population of the leader's finite memory policies, and for each $m$, let $\pi^F(m)$ be the finite memory follower's response policy to the leader's policy $\pi^L(m)$.  Further, let $v^L_i(\pi^L(m), \pi^F(m))$ be the expected cost of the leader's $i^{th}$ criterion, given the policy pair $(\pi^L(m), \pi^F(m))$. 

\begin{definition}Policy $\pi^L$ is said to \textit{dominate} policy $\rho^L$ if $$v^L_i(\pi^L, \pi^F) \leq v^L_i(\rho^L, \rho^F), \forall i$$ and there exists at least one $i$ such that $$v^L_i(\pi^L, \pi^F) < v^L_i(\rho^L, \rho^F),$$ where $\pi^F$ and $\rho^F$ are the follower's response policies to the leader's policy $\pi^L$ and $\rho^L$, respectively. 

Policy $\pi^L$ is said to be \textit{non-dominated} if there does not exist a policy that dominated policy $\pi^L$.
\end{definition}

The MOGA constructs the next generation of policies from the current set of policies as follows. The MOGA encodes each policy $\pi^L(m)$ into a chromosome. A gene is an element of the chromosome vector, and an allele is a numerical value taken by a gene. In the context of our model, the chromosome is a probability mass vector over the action space, and the $i^{th}$ gene of the chromosome denotes the probability that action $i$ is selected by the leader. The MOGA then determines $v^L_i, 1 \leq i \leq N$ for each chromosome, where $N$ is the number of leader's objectives. A description of how $v^L_i$ is determined can be found in section 3.4.  The tuple $(v^L_1,...,v^L_N)$ serves as the fitness measure of this chromosome. 

On the basis of $(v^L_1,...,v^L_N)$, the population of chromosomes are partitioned into subsets called fronts, where front $1$ is the set of non-dominated chromosomes, and front $k+1$ is the set of non-dominated chromosomes when the chromosomes in fronts $1$ through $k$ are removed from consideration, $k = 1, 2, …$.  Chromosomes in front $k$ are given rank $k$.  In addition, the crowding distance of each chromosome is determined within each front. Crowding distance is defined as the average Euclidean distance of a chromosome to the other chromosomes in the front, based on $(v^L_1,...,v^L_N)$ as a measure of position. Crowding distance is considered a measure of diversity for the policies, and for this measure, larger is considered better. The current generation of policies is sorted according to ranking and crowding distance. 

Parents are selected from the current generation of policies, based on their ranks and crowding distances. Chromosomes with higher rank and larger crowding distance are selected to generate offspring with higher probability.  The selected parents form a mating pool and generate offspring using a crossover operator. For each parents pair, the crossover operator randomly exchanges a portion of genes with each other to form two new offspring.  A mutation operator is also used to maintain genetic diversity from one generation to another. This operator randomly alters a certain percentage of genes in the current generation of policies. Then the non-dominated sorting procedure is applied again on the current population and offspring population, the top M (population size) chromosomes are kept and this is the next population. The whole algorithm repeats for a certain number of iterations. 
\subsection{Equilibria}
We remark that there are two equilibrium conditions, one associated with each agent.  With respect to the follower, assume $\bar{\pi}^F$ is the \textit{perfect} memory response policy to a given leader policy $\pi^L$.  Then, results in Proposition \ref{prop1} imply that $$v^F(\mathscr{I}^F(t)) = v^F(\pi^L, \bar{\pi}^F; \mathscr{I}^F(t)) \leq v^F(\pi^L, \rho^F; \mathscr{I}^F(t))$$ for all follower policies $\rho^F$ and all $\mathscr{I}^F(t)$, where a direct application of the results of Proposition \ref{prop2} can be used to determine $v^F(\rho^L, \rho^F; \mathscr{I}^F(t))$ for any pair of leader-follower policies $(\rho^L, \rho^F)$.

With respect to the leader, we now assume $\pi^F$ is a \textit{finite-memory approximation} of the perfect memory follower's response policy to the given leader policy $\pi^L$.  Let $v^L(\pi^L, \pi^F; \mathscr{I}^L(t))$ be the vector of criterion values for the leader's multiple objectives, given $(\pi^L, \pi^F)$, and information state $\mathscr{I}^L(t)$.  Our process of determining candidate leader policies from which the leader can select the most preferred policy is intended to determine $(\pi^L, \pi^F)$ pairs so that there exists no pair $(\rho^L, \rho^F)$ such that $$v^L_i(\rho^L, \rho^F; \mathscr{I}^L(t)) \leq v^L_i(\pi^L, \pi^F; \mathscr{I}^L(t)), \forall i$$ for all $\mathscr{I}^L(t)$, where $\rho^L$ represents any leader policy and $\rho^F$ represents a finite-memory approximation of the perfect memory follower response policy to $\rho^L$.

We note that by Proposition 1 and results in (Puterman, 1994, Chapter 6), the follower policy in the first equilibrium condition is an optimal policy; hence, the follower has no incentive to deviate from this policy. With respect to the second equilibrium condition, we note that all of the follower's policies are finite memory approximations of the follower's optimal perfect memory response policy to the leader's policy. Further, the process of determining the leader policies does not guarantee that pairs ($\pi^L, \pi^F$) will be determined that satisfy the second equilibrium condition. Thus, there is no guarantee that the leader will not want to deviate from the set of resultant non-dominated leader policies. However, given a sufficient number of generations of the MOGA and a sufficiently large $\tau$ such that the finite memory follower policy is a good approximation to an optimal perfect memory follower policy, it is reasonable to be confident that the resultant leader will have little incentive to deviate from the set of non-dominated leader policies generated.  Figure \ref{flow} provides an outline of the process for generating these non-dominated leader policies. 
\begin{figure}[H]
\centering
\includegraphics[width=4in]{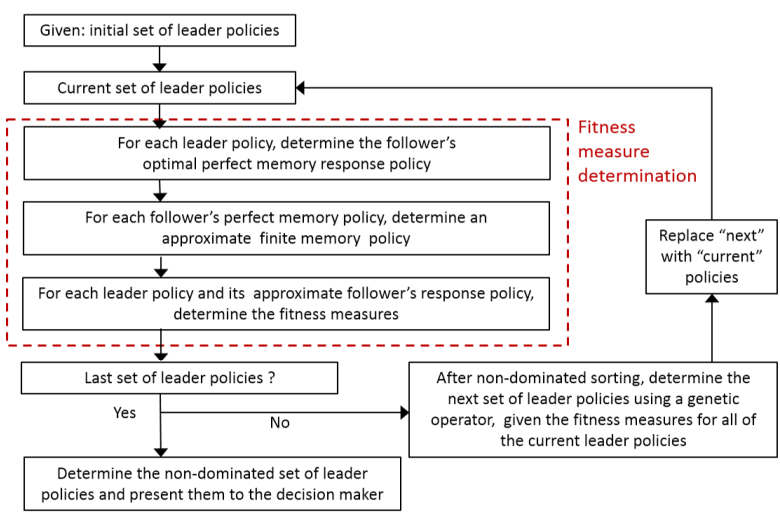}
\caption{Outline of the decision support process}
\label{flow}
\end{figure}
\subsection{Value of Information}
 We now address the question: will improved observation quality improve, or at least not degrade, the performance of agents?  With respect to the follower, assume $P(z(t+1) | s(t+1), s(t), a(t)) = P(z^F(t+1) | s^L(t+1), a(t)) P(z^L(t+1) | s^F(t+1), a(t))$, and let $P^F(a(t))$ be the stochastic matrix having $ij^{th}$ entry $P(z^F(t+1) | s^L(t+1), a(t))$, where $i = s^L(t+1)$ and $j = z^F(t+1)$. Let $Q^F(a(t))$ be a stochastic matrix such that for all $a(t)$, there exists a third stochastic matrix $R^F(a(t))$, where $Q^F(a(t)) = R^F(a(t)) P^F(a(t))$.  Then by results in (White and Harrington, 1980), $v^P(\mathscr{I}^F(t)) \leq v^Q(\mathscr{I}^F(t))$, where $v^P(\mathscr{I}^F(t))$ is the value of the follower's criterion associated with observation matrix $P^F(a(t))$  and $v^Q(\mathscr{I}^F(t))$ is the value of the follower's criterion associated with observation matrix $Q^F(a(t))$.  Thus, for the follower, the observation quality provided by $P^F(a(t))$ is at least as good as the observation quality provided by $Q^F(a(t))$ and the added value of using $P^F(a(t))$, relative to $Q^F(a(t))$, is the difference $0 \leq v^Q(\mathscr{I}^F(t)) – v^P(\mathscr{I}^F(t))$. 
  
The determination of conditions that guarantee that improved observation quality for the leader will not degrade leader performance is a topic for future research. We remark that counterexamples exist to the general claim that better information quality always implies improved system performance (Ortiz, Erera and White, 2012).
\section{Numerical Results}
This section presents an example that illustrates how the results in Section 3 can be used to support a defender (the leader) in selecting a sequence of actions to best control a simplified liquid egg production process over time, given that there is an attacker (the follower) who seeks to contaminate the liquid egg production process with a chemical or biological toxin. Both the defender and the attacker receive updated, possibly noise corrupted, data about his/her opponent just prior to each decision epoch.  The defender has two scalar criteria: (1) a measure of the system's vulnerability to an attacker, which the defender wants to minimize, and (2) a measure of the system's productivity, which the defender wants to maximize. The attacker's criterion is to maximize the expected number of packages produced by the facility that contain a sufficiently lethal dose of the toxin. 

The attacker's and defender's transition diagrams are presented in Figures \ref{attacker} and \ref{defender}, respectively. Two targets, $T_1$ and $T_2$ are considered. State $O$ is the state that the attacker is in prior to launching an attack, and includes attack team assembly, toxin manufacture, and transportation. States $PT_1$ and $PT_2$ are the pre-attack states in which the attacker is armed and ready to attack target 1 and target 2, respectively. At each decision epoch, the attacker can either choose to stay in the current state, advance forward to the next state, or retreat to a prior state. We remark that the transition probability of the attacker could be affected by the defender's strategy. An attacker's error or interdiction by the defender could return the attacker back to the state $O$.

The defender's states include a full production low alert state (FP), a low production high alert state (LP) and the attacked state (Att.). The defender in FP or LP could stay in his/her current state or switch to the other state with given probabilities. If an attack occurs, the failure of detecting an attack promptly and remain in either FP or LP will cause significant consequences for the defender. The defender can also terminate the game and shut down the facilities to clean the toxin if the attack is successfully detected. 

\begin{figure}[!t]
\centering
\includegraphics[width=2.5in]{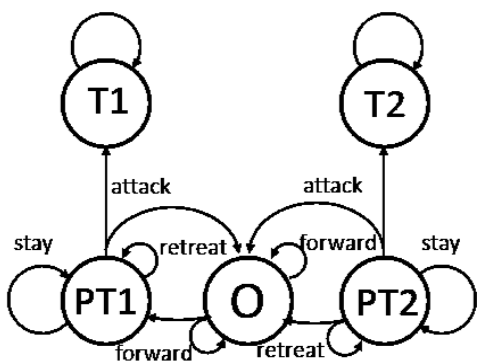}
\caption{Attacker's transition diagram}
\label{attacker}
\end{figure}
\begin{figure}[!t]
\centering
\includegraphics[width=2.5in]{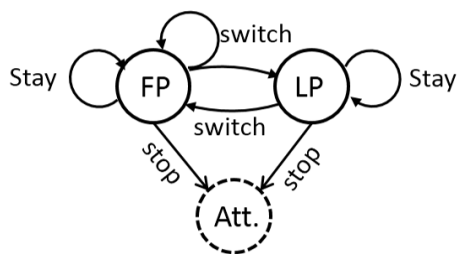}
\caption{Defender's transition diagram}
\label{defender}
\end{figure}
Given an attack, there are two groups of outcomes: unsuccessful attacks and successful attacks. If the attack is successful, the amount of toxin delivered to the consumers is described by a cumulative distribution function. We assume that the distribution function satisfies a standard, first-order stochastic dominance assumption. 

We restrict our attention to deterministic defender policies since they are easy to implement. Related discussion can be found in Paruchuri (2004).  There are 64 defender policies and the MOGA can probabilistically identify the Pareto efficient policies within 5 generations. The non-dominated set of policies are presented in Table \ref{decision_support}.  The defender must then trade off productivity and vulnerability, based on his/her preferences in order to select a most preferred policy.  
\begin{table}[H]
\caption{Decision Support Table} 
\centering 
\begin{tabular}{|c| c| c|} 
\hline 
Policy  & Productivity  & Vulnerability \\ 
& ratio to maximum & ratio to minimum\\
\hline 
$\pi_1$ & 1.000& 7.077 \\ 
$\pi_2$ & 0.959 & 5.737 \\
$\pi_3$ & 0.942 & 5.633 \\
$\pi_4$ & 0.847 & 4.136 \\
$\pi_5$ & 0.691 & 1.000 \\ 
\hline 
\end{tabular}
\label{decision_support} 
\end{table}
Figure 4 compares the value of the attacker's reward, as a function of observation quality. The results are consistent with the discussion in Section 3.7; i.e. improved observation quality of the follower does not degrade the follower's performance. A comparison of the value of the defender's rewards as a function of observation quality is a topic of future research. 
\begin{figure}[H]
\centering
\includegraphics[width=4.5in]{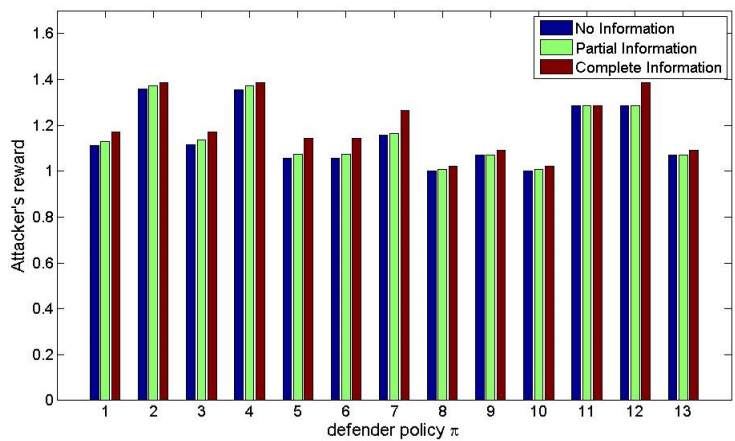}
\caption{Attacker's reward under different qualities of observations.}
\label{Att_Rwd}
\end{figure}
\section{Conclusions and Extensions}
The contributions of the paper are as follows:

(1) We have blended the POMG, the POMDP, and the MOGA to identify leader policies that are candidates for a most preferred policy in an infinite horizon, sequential decision making environment where:
\begin{itemize}
\item There are two intelligent and adaptable agents, a leader and a follower, and each can affect the performance of the other.
\item At each decision epoch, each agent knows: its past and present states, its past actions, and the noise corrupted observations of the other agent's past and present states.
\item The leader's and follower's actions are selected simultaneously at each decision epoch. 
\item Each agent's policy selects actions based on data currently available to the agent.
\item The follower knows the leader's policy and determines a response policy that is optimal with respect to the follower's objective.
\item The leader considers multiple objectives in selecting its policy.
\end{itemize}

(2) Given the POMG, a leader policy, and the assumption that the follower selects its policy with complete knowledge of and in response to the policy selected by the leader, we have constructed a specially structured POMDP that leads to the determination of a perfect-memory optimal policy for the follower (Proposition 1).  We have shown that this POMDP has a computationally useful sufficient statistic and a value function structure described in terms of this sufficient statistic.  By assuming that the leader policy is a finite-memory policy, we have shown that the sufficient statistic is finite-dimensional and that the value function has at least an approximate finite representation, thus insuring that at least a near-optimal perfect-memory policy for the follower is potentially computable.  

(3) We have determined a computationally tractable procedure for calculating the fitness measures for the MOGA, given that the policies for both agents are finite-memory policies (Proposition 2).  We have presented a simple procedure for finding a finite-memory approximation to a perfect-memory policy and used it to find a finite-memory policy for the follower, based on the perfect-memory policy determined through the use of Proposition 1.   The concomitant results show that there is a finite-dimensional sufficient statistic for the related fitness measures.

We remark that the computational tractability of procedures for determining the follower's response policy and the fitness measures for the MOGA is inextricably linked to the assumption that the agent policies are finite-memory policies.

The output of the process described in this paper can serve as the options generation phase of an option selection process; e.g., a deterministic version of multi-attribute decision theory (Kenney \& Raiffa, 1993).

Topics for future research include a sensitivity analysis in order to better understand the robustness of the results and a study of the value of improving the leader's quality of observations.  
\section*{Acknowledgments}
This material is based upon work supported by the U.S. Department of Homeland Security under Grant Award Number 2010-ST-061-FD0001 through a grant awarded by the National Center for Food Protection and Defense at the University of Minnesota. The views and conclusions contained in this document are those of the authors and should not be interpreted as necessarily representing the official policies, either expressed or implied, of the U.S. Department of Homeland Security or the National Center for Food Protection and Defense.

\end{document}